\documentclass[11pt]{amsart}

\newcommand{\cal}[1]{\mathcal{#1}}

\usepackage{john}
\usepackage{hyperref}

\numberwithin{theorem}{section}
\usepackage{color}
\usepackage{framed}

\makeatletter
\let\@@pmod\pmod
\DeclareRobustCommand{\pmod}{\@ifstar\@pmods\@@pmod}
\def\@pmods#1{\mkern4mu({\operator@font mod}\mkern 6mu#1)}
\makeatother

\newcommand{\NP}{\oper{NP}}
\newcommand{\BS}{\oper{BA}}

\newcommand{\ve}{\varepsilon}

\usepackage[displaymath, mathlines]{lineno}

\usepackage{array}
\usepackage{multirow}

\usepackage[T1]{fontenc}
\usepackage[french,english]{babel}

\usepackage[all,cmtip]{xy}
\usepackage[normalem]{ulem}
\usepackage{xcolor}
\newcommand\redsout{\bgroup\markoverwith{\textcolor{red}{\rule[0.5ex]{2pt}{1pt}}}\ULon}
\newcommand\bluesout{\bgroup\markoverwith{\textcolor{blue}{\rule[0.5ex]{2pt}{1pt}}}\ULon}
\newcommand\greensout{\bgroup\markoverwith{\textcolor{green}{\rule[0.5ex]{2pt}{1pt}}}\ULon}

\date{\today}

\address{John Bergdall\\Department of Mathematics and Statistics \\ Boston University \\ 111 Cummington Mall \\ Boston, MA 02215\\USA}
\email{bergdall@math.bu.edu}
\urladdr{http://math.bu.edu/people/bergdall}

\address{Robert Pollack\\Department of Mathematics and Statistics \\ Boston University \\ 111 Cummington Mall \\ Boston, MA 02215\\USA}
\email{rpollack@math.bu.edu}
\urladdr{http://math.bu.edu/people/rpollack}


\title{Slopes of modular forms and the ghost conjecture}
\author{John Bergdall and Robert Pollack}

\begin{document}
\begin{abstract}
We formulate a conjecture on slopes of overconvergent $p$-adic cusp forms of any $p$-adic weight in the $\Gamma_0(N)$-regular case. This conjecture unifies a conjecture of Buzzard on classical slopes and more recent conjectures on slopes ``at the boundary of weight space''.
\end{abstract}

\selectlanguage{english}

\maketitle

\section{Introduction}

Let $p$ be a prime number, and let $N$ be a positive integer co-prime to $p$. The goal of this article is to investigate $U_p$-slopes:\ the $p$-adic valuations of the eigenvalues of the $U_p$-operator acting on spaces of (overconvergent $p$-adic) cuspforms of level $\Gamma_0(Np)$. Ultimately, we formulate a conjecture which unifies currently disparate predictions for the behavior of slopes at weights ``in the center'' and ``towards the boundary'' of $p$-adic weight space.

The general study of slopes of cuspforms began with extensive computer calculations of Gouv\^ea and Mazur in the 1990's \cite{GouveaMazur-FamiliesEigenforms}. Theoretical advancements of Coleman \cite{Coleman-pAdicBanachSpaces} led to a general theory of overconvergent $p$-adic cuspforms and eventually, with Mazur, to the construction of so-called eigencurves \cite{ColemanMazur-Eigencurve}. To better understand the geometry of these eigencurves, Buzzard and his co-authors returned to explicit investigations on slopes in a series of papers \cite{Buzzard-SlopeQuestions,BuzzardCalegari-2adicSlopes,BuzzardCalegari-GouveaMazur,BuzzardKilford-2adc}. 

In \cite{Buzzard-SlopeQuestions}, Buzzard produced a combinatorial algorithm (``Buzzard's algorithm'') that for fixed $p$ and $N$ takes as input $k$ and outputs $\dim S_k(\Gamma_0(N))$-many integers.  He also defined the notion of a prime $p$ being $\Gamma_0(N)$-regular and conjectured that his algorithm was computing (classical $T_p$-)slopes in the regular cases.\footnote{Buzzard's algorithm only outputs integers, so Buzzard's conjecture implies that $T_p$-slopes are always integral in $\Gamma_0(N)$-regular cases. The authors have separately observed that $\Gamma_0(N)$-regularity is a necessary condition for the $T_p$-slopes to always be integral (\cite{BergdallPollack-FractionalSlopes}).} 
\begin{definition}\label{definition:Gamma0(N)-regular}
An odd prime $p$ is $\Gamma_0(N)$-regular if the Hecke operator $T_p$ acts on $S_k(\Gamma_0(N))$ with $p$-adic unit eigenvalues for $2 \leq k \leq p+1$.
\end{definition}
See Definition \ref{defn:2regular} for $p=2$, but we note now that $p=2$ is $\SL_2(\Z)$-regular. The first prime $p$ which is not $\SL_2(\Z)$-regular is $p=59$.

Buzzard's algorithm is concerned with spaces of cuspforms without character, where the slopes vary in a fairly complicated way with respect to the weight.  By contrast, a theorem of Buzzard and Kilford  \cite{BuzzardKilford-2adc} gives a very simple description of the $U_2$-slopes in $S_k(\Gamma_1(2^j),\chi)$ where $j \geq 3$ and $\chi$ is a primitive Dirichlet character of conductor $2^j$:\ the $i$-th slope is simply $i \cdot 2^{3-j}$. See also analogous theorems of Roe \cite{Roe-Slopes}, Kilford \cite{Kilford-5Slopes} and Kilford--McMurdy \cite{KilfordMcMurday-7adicslopes}.

In \cite{LiuXiaoWan-IntegralEigencurves}, Liu, Wan and Xiao gave  a conjectural, but general, framework in which to view the Buzzard--Kilford calculation (see \cite{WanXiaoZhang-Slopes} also). Namely, those authors have conjectured that the slopes of the $U_p$-operator acting on spaces of overconvergent $p$-adic cuspforms at $p$-adic weights ``near the boundary of weight space'' are finite unions of arithmetic progressions whose initial terms are the slopes in explicit classical weight two spaces. They also proved their conjecture for overconvergent forms on definite quaternion algebras.

The beautiful description of the slopes at the boundary of weight space is actually a consequence (\cite{BergdallPollack-FredholmSlopes,{LiuXiaoWan-IntegralEigencurves}}) of a conjecture, widely attributed to Coleman, called ``the spectral halo'':\ upon deleting a closed subdisc of weight space, the Coleman--Mazur eigencurve becomes an infinite disjoint union of finite flat covers over the remaining portion of weight space. Families of eigenforms over outer annuli of weight space should be interpreted as $p$-adic families passing through overconvergent $p$-adic eigenforms in characteristic $p$ (\cite{AndreattaIovitaPilloni-Halo, JohanssonNewton-Extended}). The existence of a spectral halo should not depend on regularity.

In summary, for a space either of the form $S_k(\Gamma_0(Np))$ or $S_k(\Gamma_1(Np^r),\chi)$ with $\chi$ having conductor $p^r$, the slopes are conjectured to be determined by a finite computation in small weights together with an algorithm:\ Buzzard's algorithm (when $p$ is $\Gamma_0(N)$-regular) in the first case and ``generate an arithmetic progression'' in the second. 

In this article, we present a unifying conjecture, which we call the ghost conjecture, that predicts the slopes of overconvergent $p$-adic cuspforms over all of $p$-adic weight space simultaneously. The shape of our conjecture is the following:\ we write down an entire series whose coefficients are functions on weight space (depending explicitly on $p$ and $N$).
We then conjecture, in the $\Gamma_0(N)$-regular case, that the Newton polygon of the specialization of our series to any given weight has the same set of slopes as the $U_p$-operator acting on the corresponding space of overconvergent $p$-adic cuspforms.

\subsection*{Organiziation}
The second section of the paper contains the statement of the conjecture when $p$ is odd, or $p=2$ and $N=1$. The third section of the paper accounts for the evidence we have compiled for our conjecture. In the fourth section, we discuss the relationship between our conjecture and the spectral halo. We include, as well, the further discovery of ``halos'' closer to the center of weight space. In the final section we discuss the case $p=2$ and $N>1$ and an irregular case.

\subsection*{Acknowledgements}
We thank Kevin Buzzard and Liang Xiao for helpful discussions. The first author was supported by NSF grant DMS-1402005 and the second author was supported by NSF grant DMS-1303302.

\section{Statement of conjecture}\label{sec:statement}

We begin with overconvergent $p$-adic cuspforms. Write $\cal  W$ for the {\em even} $p$-adic weight space:\ the space of continuous characters $\kappa: \Z_p^\times \rightarrow \C_p^\times$ with $\kappa(-1)=1$. For each $\kappa \in \cal W$ we write $S_\kappa^{\dagger}(\Gamma_0(Np))$ for the space of weight $\kappa$ overconvergent $p$-adic cuspforms of level $\Gamma_0(Np)$ (\cite{Coleman-pAdicBanachSpaces}).  The space $S_\kappa^{\dagger}(\Gamma_0(Np))$ is equipped with a (compact) operator $U_p$, and we remind the reader that the term slope refers to the $p$-adic valuation of an eigenvalue of this operator.

An integer $k$ gives rise to a $p$-adic weight $z \mapsto z^k$, and  the finite-dimensional space $S_k(\Gamma_0(Np))$ sits as a $U_p$-stable subspace of $S_k^{\dagger}(\Gamma_0(Np))$. A theorem of Coleman (\cite[Theorem 6.1]{Coleman-ClassicalandOverconvergent}) implies that classical slopes are exactly the lowest $\dim S_k(\Gamma_0(Np))$-many overconvergent slopes. Thus, one could determine the classical slopes by attempting the seemingly more difficult task of determining the overconvergent slopes.

We denote by
\begin{equation*}
P_\kappa(t) = \det\left(1 - t\restrict{U_p}{S_{\kappa}^{\dagger}(\Gamma_0(Np))}\right) = 1 + \sum_{i = 1}^\infty a_i(\kappa)t^i \in \C_p[[t]]
\end{equation*}
the Fredholm series for the $U_p$-operator in weight $\kappa$. The series $P_\kappa$ is entire in the variable $t$ and the $U_p$-slopes in weight $\kappa$ are the slopes of the Newton polygon $\NP(P_\kappa)$.  Here and below if $P=\sum_i a_i t^i$ is in $\C_p[[t]]$, $\NP(P)$ denotes the Newton polygon of $P$, i.e.\ the lower convex hull of the points $(i,v_p(a_i))$ where $v_p$ is the $p$-adic valuation normalized so that $v_p(p)=1$.

Coleman proved that  each $\kappa \mapsto a_i(\kappa)$ is defined by a power series with $\Z_p$-coefficients (see \cite[Appendix I]{Coleman-pAdicBanachSpaces}). To be precise, we write $\cal W = \bigunion_{\ve} \cal W_{\ve}$ where the (disjoint) union runs over even characters $\varepsilon:(\Z/2p\Z)^\times \rightarrow \C_p^\times$, and $\kappa \in \cal W$ is in $\cal W_{\ve}$ if and only if the restriction of $\kappa$ to the torsion subgroup in $\Z_p^\times$ is given by $\ve$. We fix a topological generator $\gamma$ for the pro-cyclic group $1+2p\Z_p$. Each $\cal W_{\varepsilon}$ is then an open $p$-adic unit disc with coordinate $w_{\kappa} = \kappa(\gamma)-1$. The meaning of Coleman's second result can now be clarified:\ for each $\varepsilon$ there exists a two-variable series
\begin{equation*}
P^{(\varepsilon)}(w,t) = 1 + \sum_{i=1}^\infty a_i^{(\varepsilon)}(w)t^i \in \Z_p[[w,t]]
\end{equation*}
such that if $\kappa \in \cal W_{\varepsilon}$ then $P_{\kappa}(t) = P^{(\varepsilon)}(w_{\kappa},t)$. In particular, the slopes of overconvergent $p$-adic cuspforms are encoded in the Newton polygons of the evaluations of the $P^{(\varepsilon)}$ at $p$-adic weights.

Our approach to predicting slopes is to create a faithful and explicit model $G^{(\varepsilon)}$ for each Fredholm series $P^{(\varepsilon)}$. We begin by writing $G^{(\varepsilon)}(w,t) = 1 + \sum g_i^{(\varepsilon)}(w)t^i$ with coefficients $g_i^{(\varepsilon)}(w)\in \Z_p[w]$ which we shortly determine. If decorations are not needed, we refer to $g(w)$ as one of these coefficients. Each coefficient will be non-zero and not divisible by $p$.\footnote{In \cite{BergdallPollack-FredholmSlopes}, the authors showed that if $N=1$ then the coefficients $a_i^{(\varepsilon)}(w)$ are not divisible by $p$. For $N >1$ this is not true, but we don't believe this divisibility plays a crucial role for predicting slopes.} In particular, $w_{\kappa} \mapsto v_p(g(w_{\kappa}))$ will depend only on the relative position of $w_{\kappa}$ to the finitely many roots of $g(w)$ in the open disc $v_p(w) > 0$.

To motivate our specification of the zeros of $g_i^{(\varepsilon)}(w)$, we make two observations:
\begin{enumerate}
\item If $g_i^{(\varepsilon)}(w_\kappa) = 0$ then the $i$-th and $(i+1)$-st slope of the Newton polygon of $G^{(\varepsilon)}(w_\kappa,t)$ are the same.
\end{enumerate}
Indeed, if $g_i^{(\varepsilon)}(w_\kappa) = 0$ then the $i$-th point of the Newton polygon in weight $\kappa$ is placed at infinity forcing a line segment of length at least  2 to appear.  So, one can ask:\ what are the slopes that appear with multiplicity greater than 1 in spaces of overconvergent $p$-adic cuspforms? The second observation is:
\begin{enumerate}
\setcounter{enumi}{1}
\item If $k\geq 2$ is an even integer then the slope ${k-2\over 2}$ is often repeated in $S_k^{\dagger}(\Gamma_0(Np))$.
\end{enumerate}
In fact,  any eigenform in $S_k(\Gamma_0(Np))$ which is new at $p$ has slope ${k-2\over 2}$. 

Combining observations (a) and (b), it might be reasonable to insist that $g_i^{(\varepsilon)}(w)$ has a zero exactly at $w = w_{k}$ with $k \in \cal W_{\varepsilon}$ where the $i$-th and $(i+1)$-st slope of $U_p$ acting on $S_k(\Gamma_0(Np))$ are both ${k-2\over 2}$. This leads us to seek $g_i^{(\ve)}$ such that:
\begin{multline}\label{eqn:mi_positive}
g_i^{(\varepsilon)}(w_{k}) = 0\\
\iff \dim S_k(\Gamma_0(N)) < i < \dim S_k(\Gamma_0(N)) + \dim S_k(\Gamma_0(Np))^{p-\new}
\end{multline}
for $k \in \cal W_{\ve}$. Such a $g_i^{(\varepsilon)}$ exists because for fixed $i$, the right-hand side of \eqref{eqn:mi_positive} holds for at most finitely many $k$. The na\"ive idea of only giving the coefficients $g_i^{(\varepsilon)}$ simple zeros would not result in an entire power series (see Lemma \ref{lemma:entirety}), so now we need to specify the multiplicities of the zeros $w_k$.

To this end, first note that an integer $k \in \cal W_{\varepsilon}$ is a zero for $g_i^{(\varepsilon)}(w)$ for some range of consecutive integers $i = a,a+1,\dots,b$ for which the right-hand side of \eqref{eqn:mi_positive} holds.  We set the order of vanishing of $g_a^{(\varepsilon)}(w)$ and $g_b^{(\varepsilon)}(w)$ at $w = w_{k}$  to be 1; for $g_{a+1}^{(\varepsilon)}(w)$ and $g_{b-1}^{(\varepsilon)}(w)$ to be 2; and so on.  More formally, define the sequence $s(\ell)$ by
\begin{equation*}
s_i(\ell) = \begin{cases}
i & \text{if $1\leq i \leq \floor{\ell / 2}$}\\
\ell+1 - i & \text{if $\floor{\ell / 2} < i \leq \ell$},
\end{cases}
\end{equation*}
and $s(\ell)$ is the empty sequence if $\ell\leq 0$. 
For $d \geq 0$ we write $s(\ell,d)$ for the infinite sequence
\begin{equation*}
s(\ell,d) = (\underlabel{d\text{ times}}{0,\dotsc,0},s_1(\ell),s_2(\ell),\dotsc,s_\ell(\ell),0,\dotsc).
\end{equation*}
For example, $s(5,3) = (0,0,0,1,2,3,2,1,0,0,0,\dots)$. If $k$ is an integer then set $d_k := \dim S_k(\Gamma_0(N))$ and $d_k^{\new} := \dim S_k(\Gamma_0(Np))^{p-\new}$. We then define multiplicities $(m_0(k), m_1(k), \dots) := s(d_k^{\new}-1,d_k)$, and set
\begin{equation*}
g_i^{(\varepsilon)}(w) := \prod_{k \in \cal W_{\varepsilon}} (w-w_{k})^{m_i(k)} \in \Z_p[w] \subset \Z_p[[w]]
\end{equation*}
which we note is a finite product.

\begin{definition}
The $p$-adic ghost series of tame level $\Gamma_0(N)$ on the component $\cal W_{\varepsilon}$ is
\begin{equation*}
G^{(\varepsilon)}(w,t) := 1 + \sum_{i=1}^\infty g_i^{(\varepsilon)}(w)t^i \in \Z_p[[w,t]].
\end{equation*}
For $\kappa \in \cal W_{\varepsilon} \subset \cal W$ set $G_\kappa := G^{(\varepsilon)}(w_\kappa,t)$.
\end{definition}
Before stating our conjecture (Conjecture \ref{conj:intro-ghost} below), we give an example and prove a simple result.
\begin{example}\label{example:explicit-p=2-info}
When $p=2$ and $N=1$, the ghost series begins
\begin{multline}\label{eqn:2-adic-ghost-two-terms}
G(w,t) = 1 + (w-w_{14})t + (w-w_{20})(w-w_{22})(w-w_{26})t^2 + \\
 (w-w_{26})(w-w_{28})(w-w_{30})(w-w_{32})(w-w_{34})(w-w_{38})t^3 +  \dotsb
 \end{multline}
The coefficient of $t^4$ is
 \begin{small}
 \begin{equation*}
 (w-w_{32}) (w-w_{34}) (w-w_{36})(w-w_{38})^2(w-w_{40})(w-w_{42})(w-w_{44})(w-w_{46})(w-w_{50}).
 \end{equation*}
 \end{small}
It is the first time a zero of multiplicity larger than one appears.
 
It is not hard to generalize the pattern above to see that 
\begin{equation}\label{eqn:gi-zeros}
g_i(w_k) = 0 \iff k \in \set{6i+8,\dotsc,12i-2} \union \set{12i+2}.
\end{equation}
The easiest way to understand the multiplicities of the zeros in \eqref{eqn:gi-zeros} is through the zeros and poles of $\Delta_i  = g_i/g_{i-1}$ for $i\geq 1$. The zeros and poles are always simple by the very definition of the multiplicity pattern $m_i(-)$.  We have:
\begin{enumerate}[(i)]
\item The zeros of $\Delta_i$ are $w_{k}$ where $k=8i+4, \dotsc, 12i - 2, 12i + 2$ is even.
\item The poles of $\Delta_i$ are $w_{k}$ where $k = 6i + 2, \dotsc, 8i-2$ is even.
\end{enumerate}
See Theorem \ref{theorem:actual-truth} below for a consequence of these calculations.
\end{example}

\begin{remark}
The reader interested in seeing more examples can download {\tt sage} code at \cite{Robwebsite}.
\end{remark}

\begin{lemma}\label{lemma:entirety}
$G_\kappa$ is entire for each $\kappa \in \cal W$.
\end{lemma}
\begin{proof}
Fix $\cal W_{\varepsilon}$, $G = G^{(\varepsilon)}$ and $g_i = g_i^{(\varepsilon)}$. As above, write $\Delta_i = g_i/g_{i-1}$. Define $\lambda(-)$ to be the number of zeros of  minus the number of poles. We claim that
\begin{equation}\label{eqn:lim-inf}
{\liminf_i} \lambda(\Delta_i)  = +\infty.
\end{equation}
From \eqref{eqn:lim-inf} it follows that $\lambda(g_i)/i \goto +\infty$ as $i\goto \infty$. Since the roots of $g_i$ are at $w_k \in p\Z_p$, we deduce that $G$ is entire over $\Z_p[[w]]$ in the sense of \cite[Section 1.3]{ColemanMazur-Eigencurve}. The lemma then follows by specializing weight-by-weight.

To show \eqref{eqn:lim-inf} one counts the zeros and poles of $\Delta_i$ (all of which are simple)  up to $O(1)$-terms.  For instance, if $k \in \cal W_{\varepsilon}$ then $w_k$ is a zero of $\Delta_i$ if and only if $d_k + 1 \leq i \leq d_k + \left\lfloor{d_k^{\new}/ 2} \right\rfloor$. Then, by standard formulas for $d_k$ and $d_k^{\new}$ (\cite[Section 6.1]{Stein-ModularForms}) we estimate the number of zeros of $\Delta_i$ by
\begin{equation*}
{1\over p-1}\left({12i\over \mu_0(N)} - {24i \over \mu_0(N)(p+1)}\right) + O(1) = {12i \over \mu_0(N)(p+1)} + O(1),
\end{equation*}
when $p$ is odd. Here, $\mu_0(N) = [\SL_2(\Z): \Gamma_0(N)]$. The number of poles, for $p$ odd again,  is ${12i \over \mu_0(N)p(p+1)} + O(1)$. When $p=2$ the formulas are modified by replacing $12$ by $6$, but in any case \eqref{eqn:lim-inf} follows.
\end{proof}
By Lemma \ref{lemma:entirety}, $G_\kappa(t) \in \C_p[[t]]$ is entire and so its Newton polygon has an infinite list of slopes each appearing with finite multiplicity.  We note that 
the valuation of $g_i^{(\varepsilon)}(w_\kappa)$ depends only on $\kappa$ and not on our choice of topological generator of $1+2p\Z_p$ and so $\kappa \mapsto \NP(G_\kappa)$ is independent of this choice as well.

\begin{conjecture}[The ghost conjecture]\label{conj:intro-ghost}
If $p$ is an odd $\Gamma_0(N)$-regular prime or $p=2$ and $N = 1$, then $\NP(G_\kappa) = \NP(P_\kappa)$ for each $\kappa \in \cal W$.
\end{conjecture}

One can check that the condition that either $p=2$ and $N=1$ or $p$ is an odd $\Gamma_0(N)$-regular prime is a necessary condition for the ghost conjecture to be true. For example, suppose that $p$ is odd and $4 \leq k \leq p+1$ is even. Then the multiplicity of the slope zero on $\NP(G_k)$ is at least $d_k$ nearly by definition. But, if $p$ is not $\Gamma_0(N)$-regular then $d_k$ is strictly larger then the multiplicity of the slope zero on $\NP(P_k)$.

We will eventually formulate a version of Conjecture \ref{conj:intro-ghost} for $p=2$ in Section \ref{sec:complements}. The reason for delaying the discussion is the analogous construction seems to be more complicated for $p=2$ and $N > 1$. The basic idea of a salvage may also be useful in irregular cases.

We will now address the relationship between Conjecture \ref{conj:intro-ghost} and other conjectures about slopes of modular forms (see the survey \cite{BuzzardGee-Slopes} for details on the specific conjectures mentioned below). The next two sections deal with how our conjecture seems to encompass both Buzzard's conjecture (Section \ref{subsec:buzzard}) and the spectral halo conjecture (Section \ref{sec:halos}).  We have omitted, on the other hand, arguments showing that Conjecture \ref{conj:intro-ghost} implies a distributional conjecture of Gouv\^ea on the slopes of $T_p$ \cite{Gouvea-WhereSlopesAre} (compare with \cite[Question 4.10]{Buzzard-SlopeQuestions}) and an asymptotic version of a conjecture of Buzzard \cite[Question 4.9]{Buzzard-SlopeQuestions} and Gouv\^ea \cite{Gouvea-WhereSlopesAre} on the highest $T_p$-slope. This is all in the $\Gamma_0(N)$-regular case of course, except the proofs of these results are intrinsic to the ghost series in the sense that we formulate and prove analogs for the ghost series and would only appeal to the ghost conjecture to make deductions about true slopes. The nature of the proofs is rather orthogonal to the tone of this paper; hence our decision to omit them.  They will be presented elsewhere. (The same comments apply to Theorem \ref{theorem:aps-boundary}  below.) Finally, it is completely obscure from the point of view of the ghost series that the slopes should always be integers in the $\Gamma_0(N)$-regular case. The likely path towards proving this is to directly prove that the ghost conjecture implies Buzzard's conjecture (which {\em is} plausible; see Section \ref{subsec:buzzard} below).

We end by discussing the multiplicities $m_i(-)$. As we indicated, the choice of $k$ for which $g_i(w_k) = 0$ is completely explained by our focus on the slopes of newforms being repeated. There is no such conceptual explanation for the multiplicities $m_i(k)$. Instead, they were discovered by an explicit calculation in the case $p=2$ and $N=1$. Specifically, the authors implemented, on a computer, a formula of Koike to write down an approximation to $P = \sum a_i(w)t^i \in \Z_2[[w,t]]$ up to the $t^{20}$-term (see \cite{Robwebsite,BergdallPollack-FredholmSlopes}). For each $i$,  we expanded the coefficient $a_i(w)$ around the points $w = w_k$ as in \eqref{eqn:gi-zeros}. Then, we were pleasantly astonished to observe that exactly $m_i(k)$-many zeros of $a_i(w)$ were visibly close to $w_k$ (for instance, within $2^{-9}$ and sometimes as close as $2^{-36}$ or $2^{-81}$).\footnote{We learned of this phenomenon, in the case of $a_1(w)$ and the weight $k=14$, from an unpublished note of Buzzard.} The data from these computations can be found at the website \cite{Robwebsite}.

We did not make similar ad-hoc calculations for other primes and levels. Instead, it seems that we ``got lucky'' in that the non-conceptual portion of our construction appears to be insensitive to $p$ and $N$. One should compare this with the analogous portion of Buzzard's algorithm which is discussed in the final paragraph of \cite[Section 3]{Buzzard-SlopeQuestions}. In either case, the point seems to be that there is some basic structural feature to the $p$-adic variation of modular forms which we do not completely understand.

\section{Comparison with known or conjectured lists of slopes}\label{section:comparison}

\subsection{Buzzard's conjecture versus the ghost conjecture}\label{subsec:buzzard}
Buzzard's algorithm exploits many known and conjectural properties of slopes, such as their internal symmetries in classical subspaces, their (conjectural) local constancy in large families, and their interaction with Coleman's $\theta$-operator, to recursively predict classical slopes. The ghost conjecture on the other hand, simply motivated by slopes of $p$-new cuspforms, predicts all overconvergent $U_p$-slopes and one obtains classical slopes by keeping the first $d_k$-many. These two approaches are completely different and, yet, appear to exactly agree. We view such agreement as compelling evidence for both conjectures. 

If $G(t) \in 1 + t\C_p[[t]]$ is a power series and $d \geq 1$, then write $G^{\leq d}$ for the truncation of $G$ in degree at most $d$. Write $\BS(k)$ for the output of Buzzard's algorithm on input $k$.  

\begin{fact}\label{fact:buzzard-agreement}
If either
\begin{enumerate}
\item $N=1$ and $p\leq 4099$ and $2\leq k\leq 2050$, or
\item $2 \leq N \leq 42$,  $3 \leq p \leq 199$ and $2 \leq k \leq 400$,
\end{enumerate}
then the list of slopes of $\NP(G_k^{\leq d_k})$ is equal to $\BS(k)$.
\end{fact}
We note that Buzzard made an extensive numerical verification of his conjecture which included all weights $k\leq 2048$ for $p=2$ and $N=1$.

The careful reader will note a striking omission in the statement of Fact \ref{fact:buzzard-agreement}:\ the agreement between the ghost slopes and the output of Buzzard's algorithm was not limited to $\Gamma_0(N)$-regular cases. Namely, neither the construction of the ghost series nor Buzzard's algorithm requires any {\em a priori} regularity hypotheses and the tests we ran to check Fact \ref{fact:buzzard-agreement} were not limited to regular cases. It seems possible that someone with enough patience could even prove, without any hypothesis on $p$ and $N$, that the output of Buzzard's algorithm agrees with the classical ghost slopes. Although neither conjecture is predicting $U_p$-slopes in the irregular case, the numbers they both output could be thought of as representing the $U_p$-slopes that ``would have occurred'' if not for the existence of a non-ordinary form of low weight.

\subsection{Comparison with computations of actual slopes}
Using computer algebra systems (like {\tt sage} \cite{sagemath}  or {\tt magma} \cite{MagmaCite}) one can compute $U_p$-slopes in various spaces. Buzzard's conjecture is concerned with the (classical) slopes in classical weights, so after observing Fact \ref{fact:buzzard-agreement} we did not perform a large scale computation comparing ghost slopes against true slopes in classical weights. (Though we did spot check that the ghost conjecture is consistent with the examples computed in \cite[Section 4]{Lauder-Computations} and \cite[Section 2.2]{Vonk-SmallPrimes}.) Instead, we made some computations verifying the ghost conjecture in cases where Buzzard's conjecture does not apply.

\begin{example}\label{ex:lauder}
The algorithm introduced by Lauder in \cite{Lauder-Computations} (see \cite{Vonk-SmallPrimes} for small primes) takes in any integer weight $k$ and outputs $P_k(t) \bmod p^{M}$ for a specified integer $M$. One can take $k=0$, for example, and compare the ghost series to the characteristic series $P_0(t)$.  Note that the corresponding slopes are not directly accessible from Buzzard's conjecture (though they can be obtained via a limiting process).

For $N=1$, $k=0$ and $p\leq 23$, we computed $P_0(t) \bmod p^{200}$. For $N=1$, $k=0$ and $29 \leq p\leq 59$ we computed $P_0(t) \bmod p^{100}$. In all of these cases, the corresponding Newton polygon had the same list of slopes as the corresponding ghost series.

Note, however, that the $59$-adic ghost series should ultimately give the wrong list of slopes. So, we further computed $P_0(t) \bmod {59}^{110}$ and saw a disagreement with the ghost series. Going from $59$-adic accuracy $100$ to $110$ is not a lucky guess:\ we have a pretty good idea how to salvage the ghost conjecture in the specific case $p=59$ and $N=1$ (see Section \ref{subsec:irregular}).
\end{example}

\begin{example}
The $U_p$-slopes in classical spaces with character of conductor $p^2$ can be accessed computationally, but do not fall within the purview of Buzzard's conjecture, even by a limiting process.

For $3\leq  p\leq 29$, we computed the $U_p$-slopes in $S_2(\Gamma_1(p^2),\chi)$ where $\chi$ was sampled from among characters of conductor $p^2$, one for each component of $\mathcal W$. The list of slopes we computed agreed with the output of the ghost series.

Furthermore, the ghost conjecture in these cases was checked to be compatible with the starting values of the arithmetic progressions in Coleman's spectral halo (see Section \ref{sec:halos} below) as predicted by combining \cite[Corollary 3.12]{BergdallPollack-FredholmSlopes} with a sufficiently strong extension of \cite{LiuXiaoWan-IntegralEigencurves}. (Specifically, one would need to extend \cite{LiuXiaoWan-IntegralEigencurves} beyond definite quaternion algebras as well as improve the quantitative portion to characters of conductor $p^2$ as opposed to $p^t$ for $t$ sufficiently large.)
\end{example}

\subsection{Comparisons with known theorems on slopes}
There are a number of cases where the list of slopes of $\NP(P_{\kappa})$ have been explicitly determined. In all the cases we know of, we independently verified that the ghost series determines the same list of slopes.  The determination of the $U_p$-slopes in these cases are due to, in order, Buzzard and Calegari \cite{BuzzardCalegari-2adicSlopes}, Buzzard and Kilford \cite{BuzzardKilford-2adc}, Roe \cite{Roe-Slopes}, Kilford \cite{Kilford-5Slopes} and Kilford and McMurdy \cite{KilfordMcMurday-7adicslopes}.  
\begin{theorem}\label{theorem:actual-truth}
$\NP(G_{\kappa}) = \NP(P_{\kappa})$ in the following cases:
\begin{enumerate}
\item 
\label{part:a}
$p=2$, $N=1$, $\kappa=0$,
\item $p=2$, $N=1$, $v_2(w_{\kappa}) < 3$,
\item $p=3$, $N=1$, $v_3(w_{\kappa}) < 1$,
\item $p=5$, $N=1$, $\kappa$ of the form $z^k\chi$ with $\chi$ conductor $25$, and
\item $p=7$, $N=1$, $\kappa \in \mathcal W_0 \cup \mathcal W_2$ of the form $z^k\chi$ with $\chi$ conductor $49$.
\end{enumerate}
\end{theorem}

Rather than discuss each case, we will verify parts (a) and (b) of the theorem and leave the remainder to the intrepid reader.   For part (\ref{part:a}), we in fact claim that if $k\leq 0$ is even then $\NP(G_k)$ is exactly the Newton polygon appearing in \cite[Conjecture 2]{BuzzardCalegari-2adicSlopes}.  Proving this claim suffices as \cite[Theorem 1]{BuzzardCalegari-2adicSlopes} proves that conjecture when $k=0$.  

To see this claim, we use the details from Example \ref{example:explicit-p=2-info}. Namely, for $k$ even and negative,
\begin{align*}
v_2(\Delta_i(w_k)) 
&= 
v_2 \left({(w_{k} - w_{8i+4}) \dotsb (w_{k}-w_{12 i - 2})(w_{k} - w_{12 i + 2}) \over (w_{k} - w_{6i+2}) \dotsb (w_{k} - w_{8i-2})}\right) \\
&=
2i + v_2\left( {(k-(8i+4))\dotsb (k-(12i-2))(k-(12i+2))\over (k-(6i+2))\dotsb (k-(8i-2))}\right) \\
&= v_2\left(2^{2i}{(-k+12i+2)!(-k+6i)!\over (-k+8i+2)!(-k+8i-2)!(-k + 12i)}\right).
\end{align*}
We used $v_2(w_{k} - w_{k'}) = 2 + v_2(k - k')$ for the second equality.  Thus, the slopes of ghost series exactly match the slopes predicted in \cite[Conjecture 2]{BuzzardCalegari-2adicSlopes} as promised.

For (b), the zeros of each $g_i$ satisfy $v_2(w) \geq 3$. So, if $v_2(w_\kappa) < 3$, then $\NP(G_\kappa)$ is the lower convex hull of $\set{(i,\lambda(g_i)v_2(w_\kappa))}$ where $\lambda(g_i)$ is the number of zeros of $g_i$. From Example \ref{example:explicit-p=2-info}, if $i\geq 1$ then $\lambda(g_i) - \lambda(g_{i-1}) = \lambda(\Delta_i) = i$. It follows that, on $v_2(w_\kappa) < 3$, the slopes of $\NP(G_\kappa)$ are $\set{v_2(w_\kappa), 2v_2(w_\kappa),3v_2(w_\kappa), \dotsc}$. This is also true for $P_\kappa$ by \cite{BuzzardKilford-2adc}.

\section{Halos and arithmetic progressions}\label{sec:halos}

Coleman's spectral halo, mentioned in the introduction, is concerned with $p$-adic weights quite far away from the integers. Specifically, we refer to the spectral halo as the conjecture: 
\begin{conjecture}[The spectral halo conjecture]\label{conj:spectral-halo}
There exists a $v > 0$ such that ${1\over v_p(w_\kappa)}\NP(P_\kappa)$ is independent of $\kappa \in \cal{W}_{\varepsilon}$ if $0 < v_p(w_\kappa) < v$.
\end{conjecture}
See the introductions to \cite{LiuXiaoWan-IntegralEigencurves, BergdallPollack-FredholmSlopes} for further discussion. We note though that the constant value of ${1\over v_p(w_\kappa)}\NP(P_\kappa)$ on $\cal W_{\ve}$ is beautifully realized as the $w$-adic Newton polygon $\NP(\bar P^{(\varepsilon)})$ where $\bar P^{(\varepsilon)}$ is the mod $p$ reduction of  $P^{(\varepsilon)}$. 

It is straightforward to see that the ghost series satisfies this halo-like behavior. Indeed, the zeros of each coefficient $g(w)$ lie in the region $v_p(w_\kappa) \geq 1$ (or $v_2(w_\kappa) \geq 3$ if $p=2$). Thus, over the complement of those regions, we have $v_p(g(w_{\kappa})) = \lambda(g)v_p(w_{\kappa})$ where $\lambda(g) = \deg g$. In particular, $\kappa \mapsto {1\over v_p(w_\kappa)}\NP(G_\kappa)$ is independent of $\kappa \in \cal W_\ve$ if $0 < v_p(w_\kappa) < 1$ (and $0 < v_2(w_{\kappa}) < 3$ if $p=2$), and the constant value is equal to $\NP(\bar{G}^{(\varepsilon)})$.

Even more can be deduced from the location of the zeros of the ghost coefficients.  For $\kappa \in \cal W$, let us define $\alpha_{\kappa} := \sup_{w \in \Z_p} v_p(w_{\kappa} - w)$. Since the zeros of the ghost coefficients are all integers,  if $\kappa,\kappa'$ lie on the same component of weight space and $v_p(w_{\kappa'}-w_{\kappa}) > \alpha_{\kappa}$, then $\NP(G_{\kappa'}) = \NP(G_{\kappa})$.  In particular, for $w_\kappa \nin \Z_p$ we find that there is a small disc around $w_\kappa$ on which the entire list, rather than just a fixed finite list, of ghost slopes is constant. The simplest example is to fix $r \geq 0$ an integer and $v$ a rational number $r < v < r+1$. Then,
\begin{enumerate}
\item $\kappa \mapsto  \NP(G_{\kappa})$ is constant on the open disc $v_p(w_\kappa) = v$, and 
\item the Newton polygons vary linearly with $v$, forming ``halos''. 
\end{enumerate}
We illustrated the halos in Figure \ref{fig:gen-slopes} below where we plotted the first twenty slopes on $v_p(w_{\kappa}) = v$ for $v \nin \Z$ when $p=2$ and $N = 1$. (The omitted regions are indicated with an open circles.\footnote{We stress that the behavior of the slopes in the omitted regions will be complicated, interweaving the disjoint branches that we've drawn.}) The picture over $v_2(w_\kappa) < 3$ illustrates the result of Buzzard--Kilford \cite{BuzzardKilford-2adc}; over $3 < v < 4$  you see pairs of parallel lines hinting at extra structure in the set of slopes, and so on. 

\begin{figure}[htbp]
\begin{center}
\includegraphics[scale=.5]{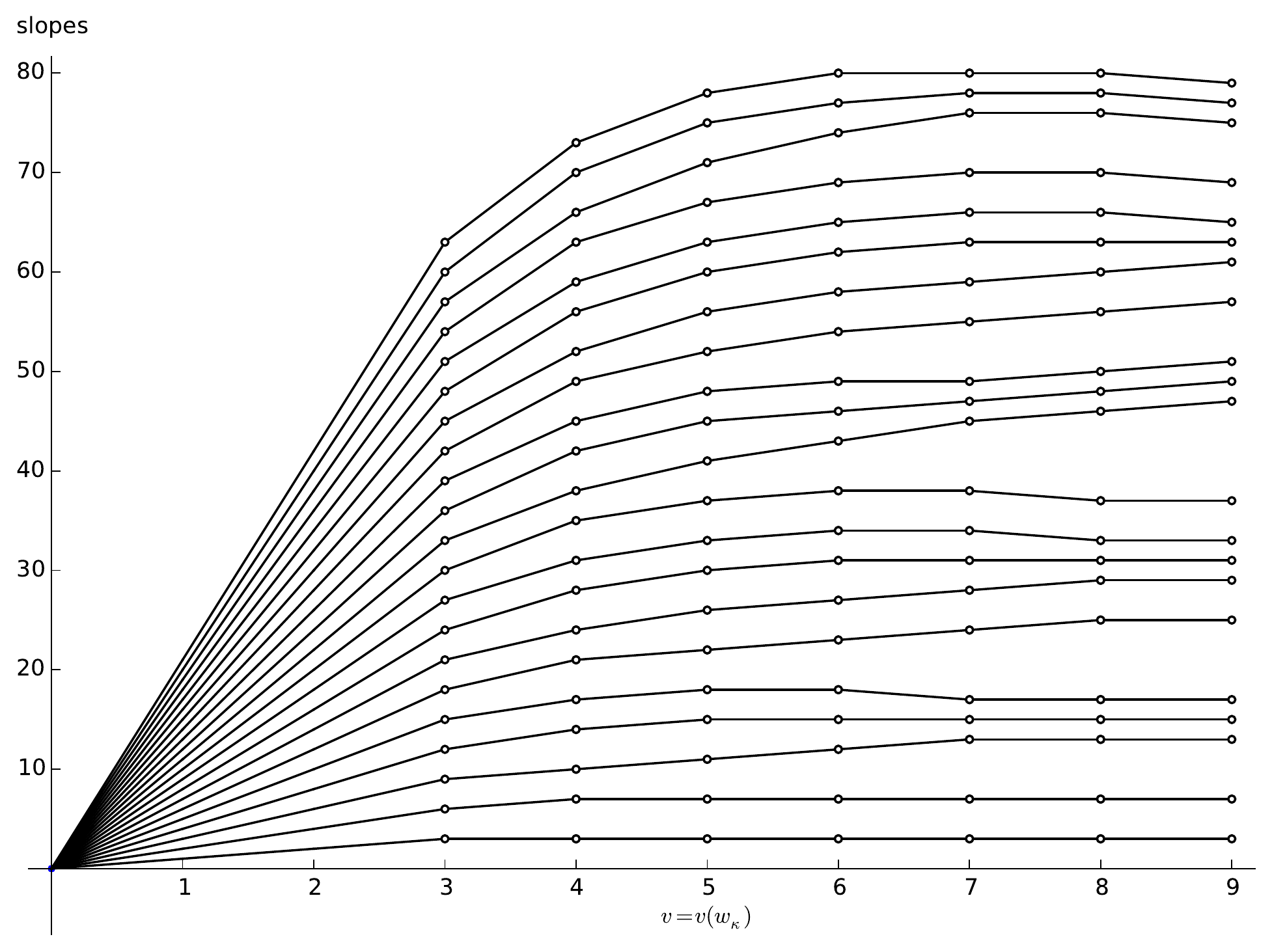}
\caption{``Halos'' centered at $k=0$ for the $2$-adic ghost series of tame level 1.}
\label{fig:gen-slopes}
\end{center}
\end{figure}

Our discussion  also implies that similar pictures may be produced on discs $v_p(w_\kappa - w_{k}) = v$ for a fixed integer $k$ and $v \nin \Z$. Figure \ref{fig:gen-slopes-2} illustrates this for $p=2$ and $N=1$ but centering the halos around $k=62$. In Figure \ref{fig:gen-slopes-2}, the thicker lines indicate higher multiplicities of slopes. For instance, the thickest line is multiplicity $6=\dim S_{62}(\Gamma_0(2))^{2-\new}$ occurring at the slope $30={62-2\over 2}$.

\begin{figure}[htbp]
\begin{center}
\includegraphics[scale=.5]{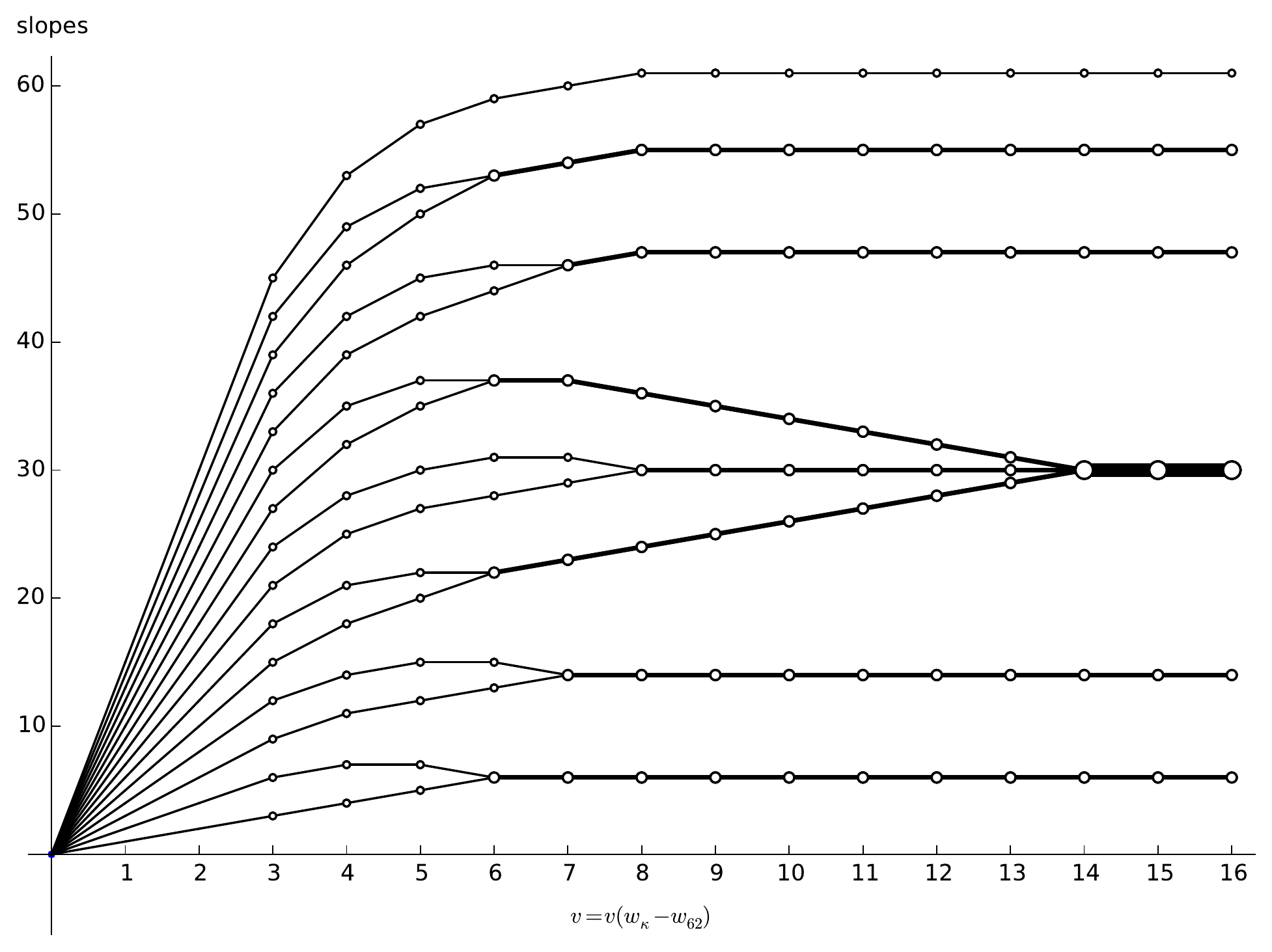}
\caption{``Halos'' centered at $k=62$ for the $2$-adic ghost series of tame level 1.}
\label{fig:gen-slopes-2}
\end{center}
\end{figure}

There is a second part to the spectral halo. Namely, one further expects that the list of slopes of $\NP(\bar P^{(\varepsilon)})$ is a finite union of arithmetic progressions (component-by-component). It is not included in the above statement since it is known to be implied by Conjecture \ref{conj:spectral-halo} (see \cite[Theorem B]{BergdallPollack-FredholmSlopes} and the proofs in \cite{LiuXiaoWan-IntegralEigencurves}). In particular, the ghost conjecture implies the slopes of the mod $p$ ghost series are a finite union of arithmetic progressions as well.

We can directly check (e.g.\ without appealing to modular forms) that ghost slopes form unions of arithmetic progressions. Specifically, we have proven the following theorem.
\begin{theorem}\label{theorem:aps-boundary}
Let $p$ be odd and write $G = G^{(\varepsilon)}$. The slopes of $\NP(\bar G)$ are a finite union of ${p(p-1)(p+1)\mu_0(N)\over 24}$-many arithmetic progressions each of whose common difference is ${(p-1)^2\over 2}$, except for finitely many possible exceptions.
\end{theorem}

We note that the number of arithmetic progressions predicted in \cite[Therorem 3.11]{BergdallPollack-FredholmSlopes} was a sum of dimensions of various spaces of classical weight 2 forms.  We checked that this sum remarkably simplifies to exactly the expression in Theorem \ref{theorem:aps-boundary} above.


We have also discovered that the previous result extends to include many more $p$-adic weights. 
Specifically, we have proven that (for odd $p$) if $\kappa$ is {\em any} weight with $w_\kappa \nin \Z_p$ then the slopes of $\NP(G_\kappa)$ are a finite union of arithmetic progressions, except for finitely many possible exceptions. The number of progressions and their common difference can be given explicitly. For instance, if $r$ is an integer and $r < v_p(w_\kappa) < r+1$ then the number of progression is $p^r$ times the number of progressions in Theorem \ref{theorem:aps-boundary}. As already discussed, the proof of this is omitted from the present paper as its details are purely combinatorial in nature.

We know of no results in this direction for the true $U_p$-slopes  (beyond results towards Theorem \ref{theorem:aps-boundary}). The situation beyond the boundary of weight space is quite mysterious:\  while over the annulus $0 < v_p(w_\kappa) < 1$ one can see Theorem \ref{theorem:aps-boundary} empirically by computing spaces of cusp forms with a character that has a large $p$-power conductor, there are no classical spaces of cuspforms over more general annuli $r < v_p(w_\kappa) < r+1$.

\section{Complements}\label{sec:complements}

We end by salvaging the ghost conjecture when $p=2$ is $\Gamma_0(N)$-regular and speculating on $\Gamma_0(N)$-irregular cases. Both discussions are motivated by the heuristics in Section \ref{sec:statement}  where we attempted to encode the ``obviously'' repeated newform slopes into the zeros of the coefficients of the ghost series.  
However, there may be more repeated slopes produced by ``pure thought'' in both the $p=2$ case and the irregular case. Specifically, if $U_p$ is acting on a space of cusp forms with a basis over $\Z_p$ then any non-integral slopes have to be repeated.  Below, we will produce fractional $2$-adic slopes in certain spaces with character when $N>1$. When $p$ is not $\Gamma_0(N)$-regular, see \cite{BergdallPollack-FractionalSlopes}.

\subsection{A modified ghost series for $p=2$}
In this subsection, $N$ is an odd positive integer.
\begin{definition}\label{defn:2regular} 
The prime $p=2$ is called $\Gamma_0(N)$-regular if:
\begin{enumerate}
\item The slopes of $T_2$ acting on $S_2(\Gamma_0(N))$ are all zero.
\item The slopes of $T_2$ acting on $S_4(\Gamma_0(N))$ are all either zero or one.
\end{enumerate}
\end{definition}
Hida theory implies that this is equivalent to \cite[Definition 1.3]{Buzzard-SlopeQuestions} and that
\begin{multline}\label{eqn:inequality-ord-dime}
\dim S_2(\Gamma_0(2N))^{\set{0}} \leq \dim S_2(\Gamma_0(N)) + \dim S_{2}(\Gamma_0(2N))^{2-\new}\\
 = \dim S_2(\Gamma_0(2N)) - \dim S_2(\Gamma_0(N))
\end{multline}
with equality if $p=2$ is $\Gamma_0(N)$-regular. Here and below, if $S$ is a space of cusp forms and $X \subset \R$ then $S^X$ is the subspace spanned by eigenforms whose slope lies in $X$.

Write $\eta_8^{\pm}$ for the (unique) Dirichlet character of conductor $8$ and sign $\pm$, and view it as a character modulo $8N$. The character $\eta_8^{\pm}$ is quadratic, so the slopes of $U_2$ acting on $S_k(\Gamma_1(8N),\eta_8^{\pm})$ are symmetric around ${k-1\over 2}$ (via the Atkin--Lehner involution; see \cite[Proposition 3.8]{BergdallPollack-FredholmSlopes}). In particular, Hida theory implies that
\begin{equation}\label{eqn:p=2_hidacons}
\dim S_2(\Gamma_1(8N),\eta_8^+)^{\set{0,1}} = 2\dim S_2(\Gamma_0(2N))^{\set{0}}.
\end{equation}

\begin{proposition}\label{prop:fractional-slopes-p=2}
If $N > 1$ and odd then $\dim S_2(\Gamma_1(8N),\eta_8^+)^{(0,1)} > 0$.
\end{proposition}
\begin{proof}
By \eqref{eqn:inequality-ord-dime} and \eqref{eqn:p=2_hidacons}, $\dim S_2(\Gamma_1(8N),\eta_8^+)^{(0,1)}$ is at least
\begin{equation*}
\dim S_2(\Gamma_1(8N),\eta_8^+) - 2\bigl(\dim S_2(\Gamma_0(2N)) - \dim S_2(\Gamma_0(N))\bigr).
\end{equation*}
One checks from standard formulas (\cite[Th\'eor\`eme 1]{CohenOesterle-Dimensions} and \cite[Section 6.1]{Stein-ModularForms} for example) that the final expression is positive for $N > 1$ and odd.
\end{proof}

Since $\eta_8^{\pm}$ is valued in $\set{\pm 1}$, the non-integral slopes in $S_k(\Gamma_1(8N),\eta_8^{\pm})$ are repeated by the general discussion at the start of this section. The ghost series defined thus far {\em does not} see these slopes:\
\begin{example}\label{example:p=2N=3}
$p=2$ is $\Gamma_0(3)$-regular. The space $S_2(\Gamma_1(24), \eta_8^+)$ is two-dimensional with 2-adic slopes $\set{{1\over 2}, {1\over 2}}$. On the other hand, one can check that the original ghost series predicts slopes $\set{0,1}$.
\end{example}

Our salvage of the ghost conjecture for $p=2$ is to  encode the fractional (repeated) slopes appearing in the spaces $S_k(\Gamma_1(8N),\eta_8^{\pm})$ as $k$ varies and $\pm = (-1)^k$ into the zeros of the coefficients of the ghost series.

Specifically, for $k\geq 2$ we write $d_k^{\circ} := \dim S_k(\Gamma_1(8N),\eta_8^{\pm})$. Write $\nu_1^{\circ}(2) \leq \nu_2^{\circ}(2) \leq \dotsb \leq \nu_{d_2^{\circ}}^{\circ}(2)$ for the list of slopes of $U_2$ acting on $S_2(\Gamma_1(8N),\eta_8^+)$; write $\mu_i$ for the multiplicity of $\nu_i^{\circ}(2)$ in this list; and write $\beta_i$ for the smallest index for which $\nu_i^{\circ}(2)$ appears as a slope.
Then set
\begin{equation*}
m_i^{\circ}(2) = \begin{cases}
s_{i}(\mu_i-1,\beta_i-1) & \text{if~}v_i^\circ(2) \not\in \Z; \\
 0 & \text{otherwise,}
\end{cases}
\end{equation*}
where $s_i(\ast,\ast)$ is the up-down pattern from Section \ref{sec:statement}.  Thus, $m_i^{\circ}(2)$ is positive if and only if the $i$-th and $(i+1)$-st slopes in $S_2(\Gamma_1(8N),\eta_8^+)$ are equal and strictly between 0 and 1. Having defined $m^{\circ}_{i}(2)$, define $m^{\circ}_{i}(k)$ for $k > 2$ by 
\begin{equation}\label{eqn:p=2-mults}
m_i^{\circ}(k) = 
\begin{cases}
m_{d_k^{\circ}-i}^{\circ}(2) & \text{if $1 \leq i < d_k^{\circ}$;}\\
0 & \text{otherwise}
\end{cases}
\end{equation}
and
\begin{equation*}
g_i^{\circ}(w) = g_i(w) \cdot \prod_{k=2}^\infty (w - w_{z^k\eta_8^{\pm}})^{m_i^{\circ}(k)} \in \Z_2[w].
\end{equation*}
Note that since $\eta_8^{\pm}(5) = -1$, the zeros of $g_i^{\circ}$ which are not zeros of $g_i$ all satisfy $v_2(w_{z^k\eta_8^{\pm}}) = 1$. Thus these zeros are far away from the zeros of $g_i$ (which lie in $v_2(w) \geq 3$). See below for further commentary.
\begin{definition}
$G^{\circ}(w,t) = 1 + \sum g_i^{\circ}(w)t^i \in \Z_2[[w,t]]$.
\end{definition}
We note that $G^{\circ}(w,t)$ is an entire series over $\Z_2[[w]]$ just as in Lemma \ref{lemma:entirety}. If $N = 1$, then $G = G^{\circ}$. 
\begin{example}
For $N=3$ (continuing Example \ref{example:p=2N=3}), the added multiplicities are mostly zero. The extra non-zero multiplicities are $m_1^{\circ}(2) = m_5^{\circ}(3) = m_9^{\circ}(4)=\dotsb=1$.  It is worth comparing the original boundary Newton polygon $\NP(\bar G)$ and the modified one $\NP(\bar G^{\circ})$ in this case. The slopes of $\NP(\bar G)$ are  $0,1,1,1,1,2,2,2,2,3,\dotsc$ whereas the slopes of $\NP(\bar G^{\circ})$ are ${1\over 2},{1\over 2},1,1,{3\over 2},{3\over 2},2,2,{5\over 2},{5\over 2},\dotsc$. This latter list matches numerical computations of the true boundary slopes (i.e.\ the slopes of $\NP(\bar P)$). Combining the results of \cite{BuzzardKilford-2adc} and \cite{LiuXiaoWan-IntegralEigencurves}, one could even prove $\NP(\bar P) = \NP(\bar G^{\circ})$ with only a finite computation. (This would give another part to add to Theorem \ref{theorem:actual-truth}.)
\end{example}

The ghost conjecture for $p=2$ is:\
\begin{conjecture}\label{conj:p=2}
If $p=2$ is $\Gamma_0(N)$-regular then $\NP(G_\kappa^{\circ}) = \NP(P_\kappa)$ for each $\kappa \in \cal W$.
\end{conjecture}
Recall we write $\BS(k)$ for the output of Buzzard's algorithm in weight $k$.
\begin{fact}\label{fact:p=2_evidence}
If $N \in \set{3,7,23,31}$ then $\NP((G_k^{\circ})^{\leq d_k}) = \BS(k)$ for all even $k\leq 5000$, or if $N \in \set{47, 71, 103, 127, 151, 167}$ then $\NP((G_k^{\circ})^{\leq d_k}) = \BS(k)$ for all even $k\leq 2050$.
\end{fact}

The $N$ listed in Fact \ref{fact:p=2_evidence} are all the levels up to $167$ such that $p=2$ is $\Gamma_0(N)$-regular.\footnote{The next such $N$ is $191$. Given this data, it is natural to ask whether $2$ being $\Gamma_0(N)$-regular implies that $N$ is either 1, 3 or  a prime congruent to $7 \bmod 8$. Anna Medvedovsky tells us that $p=2$ is not $\Gamma_0(\ell)$-regular when $\ell > 3$ is a prime $3\bmod 8$.} The bottleneck for further testing is that one needs to compute the $U_2$-action on the space $S_2(\Gamma_1(8N),\eta_8^+)$. 

We also carried out an analog of Example \ref{ex:lauder}. Specifically for the regular levels $N\leq 50$ we checked that the modified ghost slopes matched the slopes computed by $P_0(t) \bmod p^{100}$.

\begin{remark}\label{remark:we-miss-buzzard}
Contrary to Fact \ref{fact:buzzard-agreement}, the numerical agreement between the ghost conjecture and Buzzard's conjecture is actually limited to the $\Gamma_0(N)$-regular case. We do not understand why. 

For instance, $p=2$ is not $\Gamma_0(5)$-regular. Buzzard's algorithm predicts in $S_8(\Gamma_0(5))$ that the $T_2$-slopes are $\set{1,1,2}$. The modified ghost series predicts that the $T_2$-slopes are $\set{1,{3\over 2},{3\over 2}}$. It turns out that this is actually the list of $T_2$-slopes in $S_8(\Gamma_0(5))$ dubiously suggesting our method is more correct. However, both Buzzard's algorithm and our modified series predict slopes $\set{1,2,2}$ for the action of $T_2$ on $S_{10}(\Gamma_0(5))$.  The real slopes are $\set{1,2,3}$.  We are less correct in this case because $G^{\circ}$ predicts  the $U_2$-slopes in $S_{10}(\Gamma_0(5))$ as being $1, 2, 2, {9\over 2}, {9\over 2},\dotsc$, getting $2$-new slopes incorrect.
\end{remark}

Because of Remark \ref{remark:we-miss-buzzard}, let us attempt to justify our definition using the spectral halo. Here, we will take it to mean that $k\mapsto \NP(P_{z^k\eta_8^{\pm}})$ is independent of $k$ (this is a  consequence of a quantitative version of Conjecture \ref{conj:spectral-halo}). Thus, a non-redundant attempt to predict the fractional slopes in $S_k(\Gamma_1(8N),\eta_8^{\pm})$ should focus on those  slopes strictly between $k-2$ and $k-1$. For $k=2$, this is exactly what $m_i^{\circ}(2)$ does. For larger $k$, write $\nu_1^{\circ}(k) \leq \nu_2^{\circ}(k) \leq  \dotsb$ for the slopes of $U_2$ acting on $S_k(\Gamma_1(8N),\eta_8^{\pm})$. Then,
\begin{multline}\label{eqn:constraint-i}
\nu_i^{\circ}(k) = \nu_{i+1}^{\circ}(k) \text{ is in $(k-2,k-1)$}\\ \iff \nu_{d_k^{\circ}-i}^{\circ}(k) = \nu_{d_k^{\circ}-i+1}^{\circ}(k)  \text{ is in $(0,1)$}
\end{multline}
(by the Atkin--Lehner involution). By ``the spectral halo'', the right-hand side of \eqref{eqn:constraint-i} is equivalent to $\nu_{d_k^{\circ}-i}^{\circ}(2) = \nu_{d_k^{\circ}-i+1}^{\circ}(2)$ is in $(0,1)$, except the index $j=d_k^{\circ} - i$ is not always a weight two classical index $j < d_2^{\circ}$. We claim, however, that it should be provided the left-hand side of \eqref{eqn:constraint-i} holds. Indeed, every form in $S_{k-1}(\Gamma_1(8N),\eta_8^{\pm})$ and the $c_0(N)$-many non-classical $\theta^{k-2}$-critical Eisenstein series will produce slopes at most $k-2$ in weight $z^{k-1}\eta_8^{\pm}$. These slopes will also appear in weight $z^k\eta_8^{\pm}$ by ``the spectral halo''. So, if $\nu_i^{\circ}(k)$ is larger than $k-2$, then we should expect that $d_{k-1}^{\circ} + c_0(N) < i$. It turns out that this is equivalent to $d_k^{\circ}-i < d_2^{\circ}$ by a dimension count similar to the one alluded to in Proposition \ref{prop:fractional-slopes-p=2}.

\subsection{Irregular cases}\label{subsec:irregular}
We mentioned above that if $p$ is not $\Gamma_0(N)$-regular, there always exist non-integral slopes in classical spaces $S_k(\Gamma_0(N))$. The key to salvaging the ghost conjecture in the $\Gamma_0(N)$-irregular case then is being able to predict with {\em only a finite computation} 
these repeated slopes.

We examined carefully the case where $p=59$ and $N=1$ and produced a modification of the ghost series which entails adding an extra zero to infinitely many of its coefficients.
Our modified ghost series correctly predicted the $T_{59}$-slopes for weights $2 \leq k \leq 1640$. This weight range includes data significantly beyond what was computed by Gouv\^ea in \cite{Gouvea-WhereSlopesAre} (see \cite{Robwebsite} for the extended data). We note though that the focus of our modification is not on the data. Rather, we observed systematic fractional slopes appearing in spaces of cuspforms with nebentype of conductor $59$, and these observations dictated the placement of the additional zeros (compare with \cite{BergdallPollack-FractionalSlopes}).  In particular, we are not simply artificially data fitting.

However, we do not make a conjecture here for two reasons. First, computing actual slopes is computationally difficult, and we feel that there is still not enough data to support making a conjecture. Note that we cannot compare to Buzzard's algorithm in irregular cases. Second, our modification seems to have relied on a custom calculation followed by a series of coincidences that we cannot explain in much larger generality.

\bibliography{ghost_bib}
\bibliographystyle{abbrv}

\end{document}